\documentclass[USenglish]{article}

\usepackage[utf8]{inputenc}
\usepackage[big]{dgruyter}
\usepackage{microtype}

\usepackage{mathtools,caption,dsfont,float,enumerate,hyperref}
\mathtoolsset{showonlyrefs=true} 
\usepackage{natbib}
\setcitestyle{numbers,square}

\newtheorem{thm}{Theorem}[section]
\newtheorem{defin}[thm]{Definition}

\newtheorem{lem}[thm]{Lemma}
\newtheorem{rem}{Remark}[section]

\newtheorem{condition}{Condition}

\makeatletter \@addtoreset{equation}{section} \makeatother
\renewcommand\theequation{\thesection.\arabic{equation}}
\renewcommand\thefigure{\thesection.\arabic{figure}}

\newcommand{\N}{\mathbb{N}}
\newcommand{\R}{\mathbb{R}}
\newcommand{\PP}{\mathbb{P}}
\newcommand{\E}{\mathbb{E}}
\newcommand{\Var}{\mathbb{V}\mathrm{ar}}
\newcommand{\leqdef}{\vcentcolon=}

\newcommand{\ind}{\mathds{1}}
\newcommand{\ii}{\mathrm{i}}
\newcommand{\bb}[1]{\boldsymbol{#1}}

\allowdisplaybreaks

\begin{document}

    \articletype{Research Article{\hfill}Open Access}

    \title{\huge Counterexamples to the classical central limit theorem for triplewise independent random variables having a common arbitrary margin}
    \runningtitle{Counterexamples to the classical CLT}

    \author[1]{Guillaume Boglioni Beaulieu}%
    \author[2]{Pierre Lafaye de Micheaux}%
    \author*[3]{Fr\'ed\'eric Ouimet}%

    \runningauthor{G.\ Boglioni, P.\ Lafaye de Micheaux and F.\ Ouimet}

    \affil[1]{
    UNSW Sydney, NSW 2052, Australia. E-mail: g.boglioni@unsw.edu.au}%
    \affil[2]{
    UNSW Sydney, NSW 2052, Australia. E-mail: lafaye@unsw.edu.au ;
        Desbrest Institute of Epidemiology and Public Health, Univ Montpellier, INSERM, Montpellier, France ; AMIS, Université Paul Valéry Montpellier 3, France}%
    \affil[3]{
    McGill University, Montreal, QC, Canada. E-mail: frederic.ouimet2@mcgill.ca}%

    \maketitle

    \begin{abstract}a 
        \noindent
        {\bf Abstract: } We present a general methodology to construct triplewise independent sequences of random variables having a common but arbitrary marginal distribution $F$ (satisfying very mild conditions). For two specific sequences, we obtain in closed form the asymptotic distribution of the sample mean. It is  \textit{non-Gaussian} (and depends on the specific choice of $F$). This allows us to illustrate the extent of the `failure' of the classical central limit theorem (CLT) under triplewise independence. Our methodology is simple and  can also be used to create, for any integer $K$, new $K$-tuplewise independent sequences that are not mutually independent. For $K \geq 4$, it appears that the sequences created using our methodology \textit{do} verify a CLT, and we explain heuristically why this is the case.
    \end{abstract}


    \DOI{DOI}
    \startpage{1}

    \journalyear{2021}

\section{Introduction}\label{sec:introduction}

\textit{Independence} is a fundamental concept in probability.
When speaking of `independence', one generally means \textit{mutual} independence, as opposed to \textit{pairwise} independence, or, in general, `$K$-tuplewise independence' ($K \geq 2)$. Recall that a collection of random variables (defined on the same probability space) are mutually independent, or just independent, if they are $K$-tuplewise independent for all positive integers $K$.

While mutual independence implies $K$-tuplewise independence (for any $K$), the converse is not true.
For the case $K=2$ (`pairwise independence'), several counterexamples can be found in the literature, see, e.g., \citet{MR4208980} for a recent survey. For instance, one can define the following simple example
\begin{equation}
X_{3j+1} = Y_j, \qquad X_{3j+2}=Z_j, \qquad X_{3j+3} = Y_j Z_j \qquad\text{ for }j=0,1,\ldots
\end{equation}
where $Y_0,Z_0,Y_1,Z_1,\ldots$ are independent and identically distributed (i.i.d.) with $P(Y_0=1)=P(Y_0=-1)=1/2$.
Building examples of $K$-tuplewise independent variables which are not mutually independent for $K = 3$ (henceforth `triplewise independence') or $K \geq 4$ is not easy, and such examples are   scarce. This may explain why we still have an incomplete understanding of which fundamental theorems of mathematical statistics `fail' under this weaker assumption (and to what extent).
By a well known result of \citet{etemadi1981}, the classical strong law of large numbers does hold for any pairwise independent and identically distributed sequence $\{X_n\}_{n\geq 1}$ such that $\E|X_1| < \infty$.

The same is not true, though, of the classical CLT, arguably one of the most important results in all of statistics. Few authors have studied this question. \citet{pruss1998} showed that, for any integer $K$, one can build a sequence of $K$-tuplewise independent  r.v.s for which no CLT holds. \citet{bradley2009} further showed that even if such a  sequence is \textit{strictly stationary}, a CLT need not hold. \citet{weakley2013} extended this work by allowing the r.v.s in the sequence to have any symmetrical distribution (with finite variance). \citet{takeuchi2019} showed that $K$ growing linearly with the sample size $n$ is not even sufficient for a CLT to hold. In those examples, however, the asymptotic distribution of the sample mean $S_n$ is not given explicitly, hence we cannot judge to what extent it departs from normality.

\citet*{kantorovitz2007} does provide an example of a triplewise independent sequence for which $S_n$ converges to a `misbehaved' distribution ---that of $Z_1 \cdot Z_2$, where $Z_1$ and $Z_2$ are independent $N(0,1)$ --- but this is achieved for a very specific choice of margin, namely the Bernoulli distribution.

In Section~\ref{sec:construction}, we present a methodology, borrowing elements from graph theory, to construct new sequences of {\it triplewise} independent and identically distributed (noted thereafter t.i.i.d.) r.v.s\ whose common marginal distribution $F$ can be chosen arbitrarily (under very mild conditions). In  Section~\ref{sec:asymptotic-distribution-of-the-mean}, we provide a necessary and sufficient condition for a CLT to hold for such   sequences.

In Section~\ref{sec:properties-of-s}, we provide what we believe to be the first two examples of triplewise independent sequences with arbitrary margins for which the asymptotic distribution of the standardized sample mean is explicitly known and non-Gaussian. Those two distributions depend on the choice of the margin $F$ and have heavier tails than a Gaussian. This allows us to assess how far away from the Gaussian distribution one can get under sole {\it triplewise} independence. This work thus highlights why {\it mutual} independence is so fundamental for the classical CLT to hold.

Lastly, in Section~\ref{sec:why.cannot.generalize}, we explain how our methodology can easily be extended to create new $K$-tuplewise independent sequences (which are not mutually independent) for any integer $K$. While such sequences are interesting in themselves, it appears that for $K \geq 4$ they \textit{do} verify a CLT, and we explain heuristically why this is the case. Despite  not being the focus of this paper, we note that these sequences could prove  useful to benchmark the performance of multivariate independence tests, many of which have been proposed in recent years, see, e.g., \citet{fan2017, jin2018, yao2018, bottcher2019, chakraborty2019, genest2019, drton2020}.

\section{Construction of triplewise independent sequences }\label{sec:construction}

In this section, we present a general methodology to construct sequences $\{X_j\}_{j \geq 1}$ of t.i.i.d.\ r.v.s having a common (but arbitrary) marginal distribution $F$ satisfying the following condition:

\begin{condition}\label{cond:F}
    $F$ has finite variance and for any r.v.\ $W\sim F$, there exists a Borel set $A$ with $\PP(W\in A) = \ell^{-1}$, where $\ell \geq 2$ is an integer.
\end{condition}

We begin our construction of the sequence $\{X_j\}_{j \geq 1}$ by letting $F$ be a distribution satisfying Condition~\ref{cond:F}, with mean and variance denoted by  $\mu$ and $\sigma^2$, respectively. For a r.v.\ $W \sim F$,
let $A$ be any Borel set such that
\begin{equation}\label{eq:w.A}
    \PP(W\in A) = \ell^{-1}, \quad \text{for some integer } \ell\geq 2.
\end{equation}
Our construction relies on a sequence of simple graphs $\{G_m\}_{m \geq 1}$ with two properties:
\begin{enumerate}
    \item The girth of $G_m$ is 4 (or larger), for all $m$;
    \item The number of edges of $G_m$ grows to infinity as $m \to \infty$.
\end{enumerate}
Aside from these properties, the sequence $\{G_m\}_{m \geq 1}$ is left unspecified, making our construction very general. As a concrete example, consider a {\it complete bipartite graph} composed of two sets of $m$ vertices, where every vertex from one set is linked by an edge to every vertex in the second set; see Figure~\ref{fig:K_4_4} with $m=4$ for an illustration. Such graphs are often denoted by $K_{m,m}$, see, e.g, \citet[p.17]{MR2159259}.

\begin{figure}[ht]
    \centering
    \includegraphics[width=50mm]{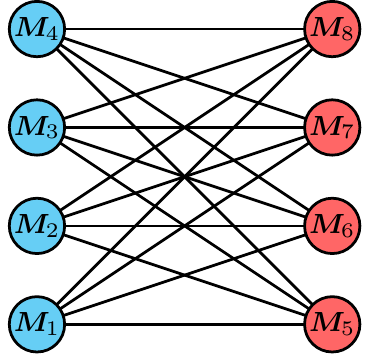}
    \caption{Graph $K_{4,4}$ with uniform r.v.s $M_j, \, 1 \leq j \leq 8,$ defined in \eqref{eq:uniform} assigned to the vertices. The vertices on the left (colored in blue) belong to one set while the vertices on the right (colored in red) belong to another set.}
        \label{fig:K_4_4}
\end{figure}

Let $v(m)$ be the number of vertices of $G_m$ and let $M_1, \dots, M_{v(m)}$ be a sequence of i.i.d.\ discrete uniforms on the set $\{1,2,\dots,\ell\}$, defined on a common probability space $(\Omega, \mathcal{F}, \mathbb{P})$. Precisely, for $i = 1,2, \dots, \ell$, let
\begin{equation}\label{eq:uniform}
    p_i \leqdef \mathbb{P}(M_1 = i) = \ell^{-1}.
\end{equation}

\noindent
Assign the uniform r.v.s $M_1, \dots, M_{v(m)}$ to the $v(m)$ vertices of the graph  (the order does not matter).  Then, for every pair $(i,j), 1 \leq i < j \leq v(m)$ such that an edge connects  $M_i$ and $M_j$, define a r.v.\ $D_{i,j}$ as
\vspace{-2mm}
\begin{equation}\label{eq:define.Dij}
    D_{i,j} =
    \begin{cases}
    1, & \textup{if } M_i = M_j, \\
    0, & \textup{otherwise}.
    \end{cases}
\end{equation}
Let $n$ be the total number of edges. For convenience, we relabel the $n$ random variables in the sequence $\{D_{i,j}\}$ simply as
\begin{align}\label{the_D_sequence}
    D_1, \dots, D_n.
\end{align}
We define $\Xi_n$ to be the number of $1$'s in the sequence $\{D_k\}_{1 \leq k \leq n}$, and $\xi_n$ its standardized version, i.e.,
\begin{align}\label{sec:define.Xi_n}
    \Xi_n = \sum_{k=1}^n D_k, \qquad \xi_n =  \frac{\Xi_n - n \ell^{-1}}{\sqrt{n \ell^{-1} (1-\ell^{-1})}}.
\end{align}

The sequence $D_1, \dots, D_n$ is triplewise independent (see Remark~\ref{rem:threeconds}) and from it we now construct a new triplewise independent sequence $X_1,\dots,X_n$ such that $X_k \sim F$, for all $k = 1,\dots,n$. Define $U$ and $V$, with cumulative distribution functions $F_U$ and $F_V$ respectively, to be the truncated versions of $W$, respectively off and on the set $A$:
\begin{equation}
    U \stackrel{\mathrm{law}}{=} W | \{W\in A^c\}, \qquad V \stackrel{\mathrm{law}}{=} W | \{W\in A\},
\end{equation}
and denote
\begin{equation}\label{eq:muU_muV}
    \mu_U \leqdef \E[U], \qquad \sigma_U^2 \leqdef \Var[U], \qquad \mu_V \leqdef \E[V], \qquad \sigma_V^2 \leqdef \Var[V].
\end{equation}
Then, consider $n$ independent copies of $U$, and independently, $n$ independent copies of $V$:
\begin{equation}\label{eq:U_V_sequence}
    U_1, U_2, \dots, U_n \stackrel{\textup{i.i.d.}}{\sim} F_U, \qquad V_1, V_2, \dots, V_n \stackrel{\textup{i.i.d.}}{\sim} F_V,
\end{equation}
both defined on the probability space $(\Omega, \mathcal{F}, \mathbb{P})$.
Finally, for $\omega \in \Omega$ and for all $k = 1,\dots, n$, construct
\begin{align}\label{the_X_sequence}
    X_k(\omega) =
    \begin{cases}
    U_k (\omega), &  \textup{if } D_k(\omega) = 0, \\
    V_k (\omega), &  \textup{if } D_k(\omega) = 1.
    \end{cases}
\end{align}
By conditioning on $D_k$, it is easy to verify that
\begin{equation}\label{eq:F.X.k}
    F_{X_k}(x) = (1 - \ell^{-1}) F_{U_k}(x) + \ell^{-1} F_{V_k}(x) = F(x).
\end{equation}

Lastly, it is not hard to see that $X_1, \dots, X_n$ is triplewise independent. Indeed, for any given $k,k',k''\in \{1,2,\dots,n\}$ with $k,k',k''$ all different, the r.v.s $D_k$, $U_k$, $V_k$, $D_{k'}$, $U_{k'}$, $V_{k'}$, $D_{k''}$, $U_{k''}$, $V_{k''}$ are mutually independent and one can write $X_k = g(D_k,U_k,V_k)$, $X_{k'} = g(D_{k'},U_{k'},V_{k'})$ and $X_{k''} = g(D_{k''},U_{k''},V_{k''})$ for $g$ a Borel-measurable function. Since $X_k$, $X_{k'}$ and $X_{k''}$ are integrable, the result follows from the triplewise independence analogue of Corollary~2 in \citet[Section 4.1]{pollard2002}.

\begin{rem}\label{rem:threeconds}
In Condition~\ref{cond:F}, the restriction $\PP(W\in A) = \ell^{-1}$ for some integer $\ell$ may seem  arbitrary. Likewise, in \eqref{eq:uniform} the choice $p_i = \ell^{-1}$ for $i=1, \dots, \ell$ may also seem arbitrary. We establish here that none of these choices are arbitrary. Indeed, assume first that the only restriction on $p_1, p_2, \dots, p_{\ell}\in (0,1)$ is that
\begin{equation}\label{eq:p1.p2.condition}
    \begin{aligned}
    &(1) ~:~ p_1 + p_2 + \dots + p_{\ell} = 1, \\
    &(2) ~:~ p_1^2 + p_2^2 + \dots + p_{\ell}^2 = w, \\
    &(3) ~:~ p_1^3 + p_2^3 + \dots + p_{\ell}^3 = w^2, \\
    &(4) ~:~ p_1^4 + p_2^4 + \dots + p_{\ell}^4 = w^3.
    \end{aligned}
\end{equation}
for some $w \in (0,1)$. Condition~$(1)$ is necessary for the distribution in \eqref{eq:uniform} to be well-defined, and conditions $(2)$, $(3)$ and $(4)$ are rewritings of
\begin{align}
    \PP(D_{v_1,v_2} = 1) &= w, \label{eq:D.condition.1} \\
    \PP(D_{v_1,v_2} = 1, D_{v_2,v_3} = 1) &= \PP(D_{v_1,v_2} = 1) \PP(D_{v_2,v_3} = 1), \label{eq:D.condition.2} \\
    \PP(D_{v_1,v_2} = 1, D_{v_2,v_3} = 1, D_{v_3,v_4} = 1)
    &= \PP(D_{v_1,v_2} = 1) \PP(D_{v_2,v_3} = 1) \PP(D_{v_3,v_4} = 1), \label{eq:D.condition.3} \\
    &\hspace{10mm} \forall \, (v_1,v_2),(v_2,v_3),(v_3,v_4)\in \mathrm{Edges}(G_m). \notag
\end{align}
(Indeed, the edges on the path $v_1 v_2 \dots v_k$ all have the value $1$ if and only if all the corresponding values on the vertices, $M_{v_1}, M_{v_2}, \dots, M_{v_k},$ are equal. With $\ell$ possible choices for each vertex, this event has probability $\PP(D_{v_{j-1},v_j} = 1 ~\forall j\in \{2,3,\dots,k\}) = \sum_{i=1}^{\ell} \prod_{j=1}^k \PP(M_j = i) = \sum_{i=1}^{\ell} p_i^k ~~\forall k\in \N$.)
Note that the conditions \eqref{eq:D.condition.1}, \eqref{eq:D.condition.2} and \eqref{eq:D.condition.3} are sufficient to guarantee that the $D$'s are identically distributed and triplewise independent. Now, the solution $p_i = \ell^{-1}$ to \eqref{eq:p1.p2.condition} is \textit{unique}. Indeed, by squaring  condition~(2) in~\eqref{eq:p1.p2.condition} then applying  the Cauchy-Schwarz inequality, one gets
\begin{equation}\label{eq:system.solve}
    w^2 = \Big(\sum_{i=1}^{\ell} p_i^{3/2} p_i^{1/2}\Big)^2 \leq \sum_{i=1}^{\ell} p_i^3 \sum_{i=1}^{\ell} p_i = \sum_{i=1}^{\ell} p_i^3
\end{equation}
where the last equality comes from condition~(1) in~\eqref{eq:p1.p2.condition}.
Then, condition~$(3)$ requires that we have the equality in \eqref{eq:system.solve}, and this happens if and only if $p_i^{\scriptscriptstyle 3/2} = \lambda p_i^{\scriptscriptstyle 1/2}$ for all $i\in\{1,\dots,\ell\}$ and for  some $\lambda\in \R$. In turn, this implies $p_i = \lambda = \ell^{-1}$ because of $(1)$ and since $p_i > 0$, which then implies $w  = \ell^{-1}$ by $(2)$.
This unique solution also satisfies $(4)$, so this reasoning shows that we cannot extend our method to an arbitrary $\PP(W\in A)\in (0,1)$ in \eqref{eq:w.A}. 
\vspace{-3mm}
\end{rem}

\section{Main result}\label{sec:asymptotic-distribution-of-the-mean}

 We now state our main result, which links the asymptotic distribution of the standardized mean of the sequence $\{X_j\}_{1 \leq j \leq n}$ to that of $\xi_n$ in \eqref{sec:define.Xi_n}. This result holds for any growing sequence of simple graphs $\{G_m\}_{m\geq 1}$ of girth at least $4$ (as defined previously).
 Specific examples are given in the next section.

\begin{thm}\label{limiting_Sn_thm}
Let $X_1, \dots, X_n$ be random variables defined as in \eqref{the_X_sequence}. Provided that there exists a r.v.\ $Y$ such that
    \begin{equation}\label{eq:convergence.standardized.chi.squared}
        \xi_n \stackrel{\mathrm{law}}{\longrightarrow} Y, \quad \text{as $m\to \infty$ (and thus as $n\to \infty$),}
    \end{equation}
then the standardized sample mean $S_n \leqdef \big(\sum_{k=1}^n X_k - n \mu\big) / \sigma \sqrt{n}$ converges in law to the random variable
            \begin{equation}\label{eq:limit.S}
                S^{(\ell)} \leqdef \sqrt{1 - r^2} Z + r \, Y,
            \end{equation}
            where $Z\sim N(0,1)$ and $r \leqdef \frac{\sqrt{\ell^{-1} (1 - \ell^{-1})} (\mu_V - \mu_U)}{\sigma}$.
\end{thm}

\begin{rem}\label{rem:main.thm.1}
If $r \neq 0$ and $\xi_n$ is asymptotically non-Gaussian (this happens for certain graphs $\{G_m\}_{m\geq 1}$, see the next section for examples), then $S_n$ is asymptotically non-Gaussian. Note that the restriction $r\neq 0$ is not stringent, as it includes {\it all} distributions $F$ (in Condition~\ref{cond:F}) with a non-atomic part. Indeed, if $W \sim F$ has a non-atomic part, then $W$ has a non-atomic part on either  $(\E[W],\infty)$ or $(-\infty,\E[W])$. Without loss of generality, assume that the non-atomic part is on $(\E[W],\infty)$, then we can find an integer $\ell\geq 2$ and a Borel set $A_0$ such that $\PP(W\in A) = \ell^{-1}$ with $A = (\E[W],\infty) \cap A_0$. By construction, this yields
\begin{equation}
    \E[W | A] > \E[W] = \E[W \ind_{A}] + \E[W \ind_{A^c}] = \E[W | A] \, \ell^{-1} + \E[W | A^c] \, (1 - \ell^{-1}),
\end{equation}
so that $\E[W | A] > \E[W | A^c]$.
The restriction $r\neq0$ also includes {\it almost all} discrete distributions with at least one weight of the form $\ell^{-1}$; see Remark~2 in \citet{MR4208980} for a formal argument. Also, note that, depending on $F$, many choices for $A$ (with possibly different values of $\ell$) could be available.
\end{rem}

\begin{rem}\label{rem:main.thm.2}
    If the margin $F$ satisfies Condition~\ref{cond:F}, and if $r = 0$ (i.e., $\mu_U = \mu_V$) or $\xi_n$ is asymptotically Gaussian, then our construction provides new triplewise independent (but not mutually independent) sequences which do satisfy a CLT (regardless of which graphs $\{G_m\}_{m\geq 1}$ are used).
\end{rem}

\begin{proof}[Proof of Theorem~\ref{limiting_Sn_thm}]
    We prove \eqref{eq:limit.S} by obtaining the limit of the characteristic function of $S_n$, and then by invoking L{\'e}vy's continuity theorem.
    Namely, we show that, for all $t\in \R$,
    \begin{align}\label{cf_S}
        \varphi_{S_n}(t) \stackrel{m\to \infty}{\longrightarrow} \varphi_{\sqrt{1 - r^2} Z}(t) \cdot \varphi_{r Y}(t).
    \end{align}
Recall the notation defined in \eqref{eq:muU_muV} and let
    \begin{equation}\label{definition_r_U_V_tilde}
        \widetilde{U}_k \leqdef \frac{\sigma_U}{\sigma} \cdot \frac{U_k - \mu_U}{\sigma_U} ~\quad \text{and} ~\quad \widetilde{V}_k \leqdef \frac{\sigma_V}{\sigma} \cdot \frac{V_k - \mu_V}{\sigma_V},
    \end{equation}
    then we can write
    \begin{equation}
        \begin{aligned}
            S_n
            &= \frac{\sum_{k=1}^n X_k - n \mu }{\sigma \sqrt{n}} = \frac{\sum_{\substack{k=1 \\ D_k = 0}}^n U_k + \sum_{\substack{k=1 \\ D_k = 1}}^n V_k - n \mu}{\sigma \sqrt{n}} \\
            &= \frac{1}{\sqrt{n}} \Bigg(\frac{(n - \Xi_n) \mu_U + \Xi_n \mu_V - n \mu}{\sigma} + \sum_{\substack{k=1 \\ D_k = 0}}^n \frac{U_k - \mu_U}{\sigma} + \sum_{\substack{k=1 \\ D_k = 1}}^n \frac{V_k - \mu_V}{\sigma}\Bigg) \\
            &= \frac{1}{\sqrt{n}} \Bigg(\frac{(\mu_V - \mu_U)}{\sigma} \bigg[\Xi_n - n \frac{(\mu - \mu_U)}{\mu_V - \mu_U}\bigg] + \sum_{\substack{k=1 \\ D_k = 0}}^n \widetilde{U}_k + \sum_{\substack{k=1 \\ D_k = 1}}^n \widetilde{V}_k\Bigg) \\
            &= \frac{1}{\sqrt{n}} \Bigg(r \, \frac{\big(\Xi_n - n \ell^{-1}\big)}{\sqrt{\ell^{-1} (1 - \ell^{-1})}} + \sum_{\substack{k=1 \\ D_k = 0}}^n \widetilde{U}_k + \sum_{\substack{k=1 \\ D_k = 1}}^n \widetilde{V}_k\Bigg),
        \end{aligned}
    \end{equation}
    since $\Xi_n = \#\{k: D_k = 1\}$, and we know that, from \eqref{eq:F.X.k},
    \begin{equation}\label{eq:mu.decomposition}
        \frac{\mu - \mu_U}{\mu_V - \mu_U} = \frac{[(1 - \ell^{-1}) \mu_U + \ell^{-1} \mu_V] - \mu_U}{\mu_V - \mu_U} = \ell^{-1}.
    \end{equation}
    With the notation $t_n \leqdef t / \sqrt{n}$, the mutual independence between the $U_k$'s, the $V_k$'s and $\bb{M} \leqdef \{M_j\}_{j=1}^{v(m)}$ yields, for all $t\in \R$,
    \vspace{-1mm}
    \begin{align}\label{eq:thm:main.result.beginning}
        \E\big[e^{\ii t S_n} | \bb{M}\big]
        &= e^{\ii t r \, \frac{(\Xi_n - n \ell^{-1})}{\sqrt{n \ell^{-1} (1 - \ell^{-1})}}} \prod_{\substack{k=1 \\ D_k = 0}}^n \E[e^{\ii t_n \widetilde{U}_k} | \bb{M}] \prod_{\substack{k=1 \\ D_k = 1}}^n \E[e^{\ii t_n \widetilde{V}_k} | \bb{M}] \notag \\[-2mm]
        &= e^{\ii t r \, \xi_n} ~ [\varphi_{\widetilde{U}}(t_n)]^{n (1 - \ell^{-1})} [\varphi_{\widetilde{V}}(t_n)]^{n \ell^{-1}} \bigg[\frac{\varphi_{\widetilde{V}}(t_n)}{\varphi_{\widetilde{U}}(t_n)}\bigg]^{\Xi_n - n \ell^{-1}} \notag \\[-1mm]
        &= e^{\ii t r \, \xi_n} \cdot [\varphi_{\widetilde{U}}(t_n)]^{n (1 - \ell^{-1})} [\varphi_{\widetilde{V}}(t_n)]^{n \ell^{-1}} \notag \\
        &\quad\cdot \left[\frac{[\varphi_{\widetilde{V}}(t_n)]^n \cdot e^{\frac{1}{2} \cdot \frac{\sigma_V^2}{\sigma^2} t^2}}{[\varphi_{\widetilde{U}}(t_n)]^n \cdot e^{\frac{1}{2} \cdot \frac{\sigma_U^2}{\sigma^2} t^2}}\right]^{\hspace{-0.5mm}\frac{\Xi_n - n \ell^{-1}}{n}} \hspace{-6mm} \cdot \hspace{4mm} \left[\frac{e^{-\frac{1}{2} \cdot \frac{\sigma_V^2}{\sigma^2} t^2}}{e^{-\frac{1}{2} \cdot \frac{\sigma_U^2}{\sigma^2} t^2}}\right]^{\hspace{-0.5mm}\frac{\Xi_n - n \ell^{-1}}{n}}\hspace{-6mm}.
    \end{align}
    (The reader should note that, for $n$ large enough, the manipulations of exponents in the second and third equality above are valid because the highest powers of the complex numbers involved have their principal argument converging to $0$. This stems from the fact that $\Xi_n \leq n$, and the quantities $[\varphi_{\widetilde{V}}(t_n)]^n$ and $[\varphi_{\widetilde{V}}(t_n)]^n$ both converge to real exponentials as $n\to \infty$, by the CLT.)
    We now evaluate the four factors on the right-hand side of \eqref{eq:thm:main.result.beginning}.
    For the first factor in \eqref{eq:thm:main.result.beginning}, the continuous mapping theorem and \eqref{eq:convergence.standardized.chi.squared} yield
    \begin{equation}\label{eq:thm:main.result.beginning.part.1}
        e^{\ii t r \, \xi_n} \stackrel{\mathrm{law}}{\longrightarrow} e^{\ii t r Y}, \quad \text{as } m\to \infty.
    \end{equation}
    For the second factor in \eqref{eq:thm:main.result.beginning}, the classical CLT yields
    \begin{equation}\label{eq:thm:main.result.beginning.part.2}
        \begin{aligned}
            [\varphi_{\widetilde{U}}(t_n)]^{n (1 - \ell^{-1})} [\varphi_{\widetilde{V}}(t_n)]^{n \ell^{-1}}
            &\stackrel{m\to \infty}{\longrightarrow} \exp\Big(-\frac{1}{2} \cdot (1 - \ell^{-1}) \frac{\sigma_U^2}{\sigma^2} t^2\Big) \exp\Big(-\frac{1}{2} \cdot \ell^{-1} \frac{\sigma_V^2}{\sigma^2} t^2\Big) \\
            &= e^{-\frac{1}{2} (1 - r^2) t^2},
        \end{aligned}
    \end{equation}
    where in the last equality we used the fact that, from \eqref{eq:F.X.k},
    \begin{equation}\label{eq:relation.r:var}
        \sigma^2 = \E[X^2] - \mu^2 = (1 - \ell^{-1}) \sigma_U^2 + \ell^{-1} \sigma_V^2 + \ell^{-1} (1 - \ell^{-1}) (\mu_U - \mu_V)^2.
    \end{equation}
    For the third factor in \eqref{eq:thm:main.result.beginning}, the quantity inside the bracket converges to $1$ by the CLT.
    Hence, the elementary bound
    \begin{equation}
        |e^z - 1| \leq |z| + \sum_{j=2}^{\infty} \frac{|z|^j}{2} \leq |z| + \frac{|z|^2}{2 (1 - |z|)} \leq \frac{1 + \ell^{-1}}{2 \ell^{-1}} |z|, ~\quad \text{for all } |z| \leq 1 - \ell^{-1},
    \end{equation}
    and the fact that $\big|\frac{\Xi_n - n \ell^{-1}}{n}\big| \leq 1 - \ell^{-1}$ yield, as $m\to \infty$,
    \begin{equation}\label{eq:thm:main.result.beginning.part.3}
        \left|\left[\frac{[\varphi_{\widetilde{V}}(t_n)]^n \cdot e^{\frac{1}{2} \cdot \frac{\sigma_V^2}{\sigma^2} t^2}}{[\varphi_{\widetilde{U}}(t_n)]^n \cdot e^{\frac{1}{2} \cdot \frac{\sigma_U^2}{\sigma^2} t^2}}\right]^{\frac{\Xi_n - n \ell^{-1}}{n}} - 1\right| \stackrel{\text{a.s.}}{\leq} \frac{1 - \ell^{-2}}{2 \ell^{-1}} \left|\text{Log}\left[\frac{[\varphi_{\widetilde{V}}(t_n)]^n \cdot e^{\frac{1}{2} \cdot \frac{\sigma_V^2}{\sigma^2} t^2}}{[\varphi_{\widetilde{U}}(t_n)]^n \cdot e^{\frac{1}{2} \cdot \frac{\sigma_U^2}{\sigma^2} t^2}}\right]\right| \longrightarrow 0.
    \end{equation}
    For the fourth factor in \eqref{eq:thm:main.result.beginning}, we note that $\frac{\Xi_n - n \ell^{-1}}{n} \stackrel{\PP}{\longrightarrow} 0$ (because of the law of large numbers for pairwise independent r.v.s). Then, by the continuous mapping theorem,
    \vspace{-2mm}
    \begin{equation}\label{eq:thm:main.result.beginning.part.4}
        \left[\frac{e^{-\frac{1}{2} \cdot \frac{\sigma_V^2}{\sigma^2} t^2}}{e^{-\frac{1}{2} \cdot \frac{\sigma_U^2}{\sigma^2} t^2}}\right]^{\frac{\Xi_n - n \ell^{-1}}{n}} \stackrel{\PP}{\longrightarrow} 1, \quad \text{as } m\to \infty.
    \end{equation}
    By combining \eqref{eq:thm:main.result.beginning.part.1}, \eqref{eq:thm:main.result.beginning.part.2}, \eqref{eq:thm:main.result.beginning.part.3} and \eqref{eq:thm:main.result.beginning.part.4}, Slutsky's lemma implies, for all $t\in \R$,
    \begin{equation}\label{eq:thm:main.result.beginning.combined}
        \E\big[e^{\ii t S_n} | \bb{M}\big] \stackrel{\mathrm{law}}{\longrightarrow} e^{\ii t r Y} e^{-\frac{1}{2} (1 - r^2) t^2}, \quad \text{as } m\to \infty.
    \end{equation}
    Since the sequence $\{|\E[e^{\ii t S_n} | \bb{M}]|\}_{m\in \N}$ is uniformly integrable (it is bounded by $1$), Theorem 25.12 in \citep{billingsley1995} shows that we also have the mean convergence
    \begin{equation}\label{eq:thm:main.result.mean.convergence}
        \E\big[\E\big[e^{\ii t S_n} | \bb{M}\big]\big] \longrightarrow \E\big[e^{\ii t r Y}\big] e^{-\frac{1}{2} (1 - r^2) t^2}, \quad \text{as } m\to \infty,
    \end{equation}
    which proves \eqref{cf_S}. The conclusion follows.
\end{proof}

\section{Examples 
}\label{sec:properties-of-s}

In Theorem~\ref{limiting_Sn_thm}, whether the standardized sample mean $S_n$ is asymptotically Gaussian depends on the ‘connectivity’ of the chosen graphs $\{G_m\}_{m\geq 1}$. In particular, it appears that having graphs of bounded diameter is a necessary (albeit not sufficient) condition for $S_n$ to be asymptotically non-Gaussian. To make this point explicit, we present two specific examples for which we obtain the (non-Gaussian) asymptotic distribution of $\xi_n$ (via Theorem~\ref{limiting_Sn_thm}, this also provides the asymptotic distribution of $S_n$). We present a third example where the limiting distribution \textit{is} Gaussian.

\subsection{First example}\label{sec:example.1}

\begin{thm}
Let $\{G_m\}_{m\geq 1}$ be the sequence of bipartite graphs $\{K_{m,m}\}_{m\geq 1}$ described above Figure~\ref{fig:K_4_4}, and consider the construction from Section~\ref{sec:construction} where i.i.d.\ discrete uniforms $M_1, \ldots, M_{2m}$ are assigned to the vertices of $G_m$. That is, $M_1,\ldots, M_m$ are assigned to the $m$ vertices of set $1$, and $M_{m+1}, \dots, M_{2m}$ to the $m$ vertices of set $2$. Then,
\vspace{-2mm}
    \begin{align}
\xi_n  &\stackrel{\mathrm{law}}{\longrightarrow}  \frac{\xi}{\sqrt{\ell-1}}, \quad \text{as $m\to \infty$ (and thus as $n\to \infty$),}
    \end{align}
where $ \xi\sim \textup{VG}(\ell-1,0,1,0),$ and \textup{VG} denotes the variance-gamma distribution (see Definition~\ref{def:variance.gamma.distribution}).
\end{thm}

\begin{rem}
    Because a standardized $\textup{VG}(\ell-1,0,1,0)$ distribution converges to a standard Gaussian as $\ell$ tends to infinity, we see that, in Theorem~\ref{limiting_Sn_thm}, $S^{(\ell)} \stackrel{\mathrm{law}}{\longrightarrow} N(0,1)$ as $\ell\to \infty$.
\end{rem}

\begin{proof}

First, note that $v(m)= 2m$ and $n=m^2$. Define, for $i\in \{1,2,\dots,\ell\}$,
    \begin{equation*}
        \begin{aligned}
            &N_i^{(1)} = N_i^{(1)}(m), ~\text{the number of $M_j$'s equal to $i$ within the sample $\{M_j\}_{j=1}^m$}, \\
            &N_i^{(2)} = N_i^{(2)}(m), ~\text{the number of $M_j$'s equal to $i$ within the sample $\{M_j\}_{j=m+1}^{2m}$}. \\
        \end{aligned}
    \end{equation*}
    Then, $\bb{N}^{(j)} \leqdef (N_1^{(j)},\dots,N_{\ell}^{(j)}) \sim \textup{Multinomial}\hspace{0.3mm}(m, (\ell^{-1}, \dots, \ell^{-1}))$ for $j\in \{1,2\}$, and $\bb{N}^{(1)}$ and $\bb{N}^{(2)}$ are independent.
    Importantly, if $\bb{N}^{(1)}$ and $\bb{N}^{(2)}$ are known, then the number of $1$'s in the sequence $\{D_k\}_{1 \leq k \leq n}$, denoted by $\Xi_n$ throughout, can be deduced from simple calculations as
    \begin{align}\label{np}
        \Xi_n
        &= \sum_{i=1}^{\ell} N_i^{(1)} N_i^{(2)}
        = \sum_{i=1}^{\ell-1} N_i^{(1)} N_i^{(2)} + \big(m - \sum_{i=1}^{\ell-1} N_i^{(1)}\big) \big(m - \sum_{i'=1}^{\ell-1} N_{i'}^{(2)}\big) \notag \\
        &= \sum_{i=1}^{\ell-1} N_i^{(1)} N_i^{(2)} + \sum_{i=1}^{\ell-1} \sum_{i'=1}^{\ell-1} N_i^{(1)} N_{i'}^{(2)} - m \sum_{i=1}^{\ell-1} N_i^{(1)} - m \sum_{i'=1}^{\ell-1} N_{i'}^{(2)} + m^2 \notag \\
        &= \sum_{i=1}^{\ell-1} (N_i^{(1)} - m \ell^{-1}) (N_i^{(2)} - m \ell^{-1}) + \sum_{i=1}^{\ell-1} \sum_{i'=1}^{\ell-1} (N_i^{(1)} - m \ell^{-1}) (N_{i'}^{(2)} - m \ell^{-1}) \notag \\[-2mm]
        &\quad + m \ell^{-1} \sum_{i=1}^{\ell-1} N_i^{(1)} + m \ell^{-1} \sum_{i=1}^{\ell-1} N_i^{(2)} + (\ell-1) m \ell^{-1} \sum_{i=1}^{\ell-1} N_i^{(1)} + (\ell-1) m \ell^{-1} \sum_{i=1}^{\ell-1} N_i^{(2)} \notag \\[-2mm]
        &\quad- m \sum_{i=1}^{\ell-1} N_i^{(1)} - m \sum_{i=1}^{\ell-1} N_i^{(2)} - (\ell-1) m^2 \ell^{-2} - (\ell-1)^2 m^2 \ell^{-2} + m^2 \notag \\
        &= \sum_{i=1}^{\ell-1} (N_i^{(1)} - m \ell^{-1}) (N_i^{(2)} - m \ell^{-1}) + \sum_{i=1}^{\ell-1} \sum_{i'=1}^{\ell-1} (N_i^{(1)} - m \ell^{-1}) (N_{i'}^{(2)} - m \ell^{-1}) + m^2 \ell^{-1} \notag \\
        &= m \ell^{-1} \sum_{i=1}^{\ell-1} \sum_{i'=1}^{\ell-1} \Big(\frac{1}{\ell^{-1}} \ind_{\{i=i'\}} + \frac{1}{\ell^{-1}}\Big) \frac{(N_i^{(1)} - m \ell^{-1})}{\sqrt{m}} \frac{(N_i^{(2)} - m \ell^{-1})}{\sqrt{m}} + m^2 \ell^{-1},
    \end{align}
    where $\ind_B$ denotes the indicator function on the set $B$.
    It is well known that the covariance matrix for the first $\ell-1$ components of a $\textup{Multinomial}\hspace{0.3mm}(m, (p_1,p_2,\dots,p_{\ell}))$ vector is $m \Sigma$ where $\Sigma_{i,i'} = p_i \ind_{\{i=i'\}} - p_i p_{i'}$, for $1 \leq i,i' \leq \ell-1$, and also that $(\Sigma^{-1})_{i,i'} = p_i^{-1} \ind_{\{i=i'\}} + p_{\ell}^{-1}, ~ 1 \leq i,i' \leq \ell-1,$ see \citet[eq.\hspace{0.5mm}21]{tanabe1992}.
    If $\Sigma = L L^{\top}$ is the Cholesky decomposition of $\Sigma$ when $p_i = \ell^{-1}$ for all $i$, and $\bb{Y}_{\hspace{-0.5mm}1} \leqdef (N_i^{(1)} - m \ell^{-1})_{i=1}^{\ell-1}$ and $\bb{Y}_{\hspace{-0.5mm}2} \leqdef (N_i^{(2)} - m \ell^{-1})_{i=1}^{\ell-1}$, then we have
    \begin{equation}
        \begin{aligned}
            \Xi_n - m^2 \ell^{-1}
            &= m \ell^{-1} \bb{Y}_{\hspace{-0.5mm}1}^{\top} (m \Sigma)^{-1} \bb{Y}_{\hspace{-0.5mm}2} \\
            &= m \ell^{-1} (m^{-1/2} L^{-1} \bb{Y}_{\hspace{-0.5mm}1})^{\top} (m^{-1/2} L^{-1} \bb{Y}_{\hspace{-0.5mm}2}).
        \end{aligned}
    \end{equation}
    By the classical multivariate CLT and Definition~\ref{def:variance.gamma.distribution} in Appendix~\ref{sec:variance.gamma}, we get the result.
\end{proof}

    Next, we illustrate what the asymptotic distribution of $S_n$ (the standardized sample mean) looks like in this example. By Theorem~\ref{limiting_Sn_thm},  $S_n$  converges in law to a r.v.\
   \begin{equation}
        S^{(\ell)} \stackrel{\mathrm{law}}{=} \sqrt{1 - r^2} Z + r \, \frac{\xi}{\sqrt{\ell - 1}},
    \end{equation}
    where the r.v.s $Z\sim N(0,1)$ and $\xi\sim \textup{VG}(\ell-1,0,1,0)$ (see Definition~\ref{def:variance.gamma.distribution} in Appendix~\ref{sec:variance.gamma}) are independent.

    For a fixed $\ell \geq 2$, the distribution of $S^{(\ell)}$ has only one parameter, $r$ (defined in Theorem~\ref{limiting_Sn_thm}), which depends on the margin $F$ (through the quantities $A$, $\mu_U$, $\mu_V$ and $\sigma$). Note that  $0 \leq r^2 \leq 1$, and that the critical points $r^2 = 0, 1$ are reachable for certain choices of $F$, see Section~4 and Appendix~A in \citet{MR4208980} for specific examples.

 Hence, when $\ell \geq 2$ is fixed, $r$ completely determines the shape of $S^{(\ell)}$; $r$ close to $0$ means that $S^{(\ell)}$ is close to a standard Gaussian, while $r$ close to $\pm 1$ means that $S^{(\ell)}$ is close to a standardized $\textup{VG}(\ell-1,0,1,0)$.
    Figure~\ref{cdf_df_r2_comparison} (where $\ell = 2$ and $r$ varies) illustrates this shift from a Gaussian distribution towards a $\textup{VG}(\ell-1,0,1,0)$ distribution. On the other hand, regardless of $r$, if $\ell$ increases then $S^{(\ell)}$ gets closer to a $N(0,1)$. This is illustrated in Figure~\ref{cdf_df_ell_comparison} (where $r=0.99$ and $\ell$ varies). It is clear from these figures that triplewise independence can be a very poor substitute to mutual independence as an assumption in the classical CLT.

    \begin{figure}[H]
        \centering
        \hspace{-5mm}
        \includegraphics[scale=0.485]{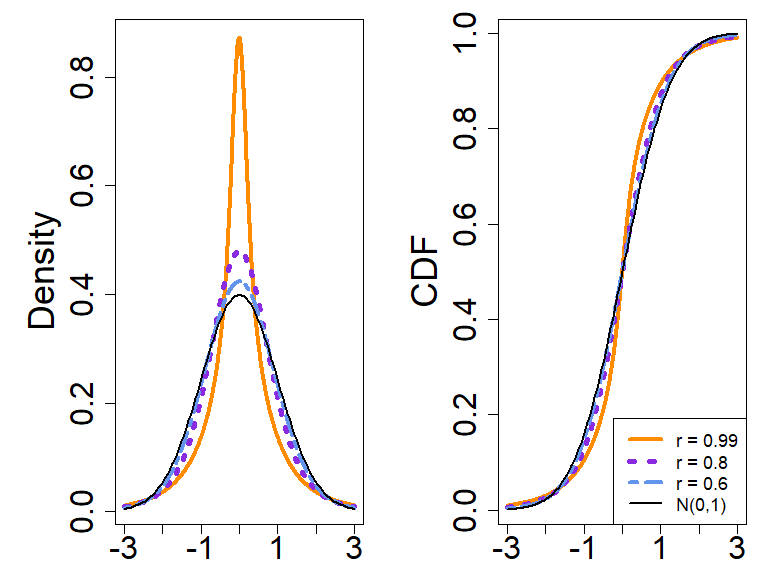}
        \caption{Density (left) and c.d.f.\  (right)  of $S^{(\ell)}$ for fixed $\ell = 2$ and varying $r$ ($r = 0.6, 0.8, 0.99)$, compared to those of a $N(0,1)$. This illustrates that the CLT can `fail' substantially under triplewise independence.}
        \label{cdf_df_r2_comparison}
        \centering
        \vspace{3mm}
        \hspace{-5mm}
        \includegraphics[scale=0.485]{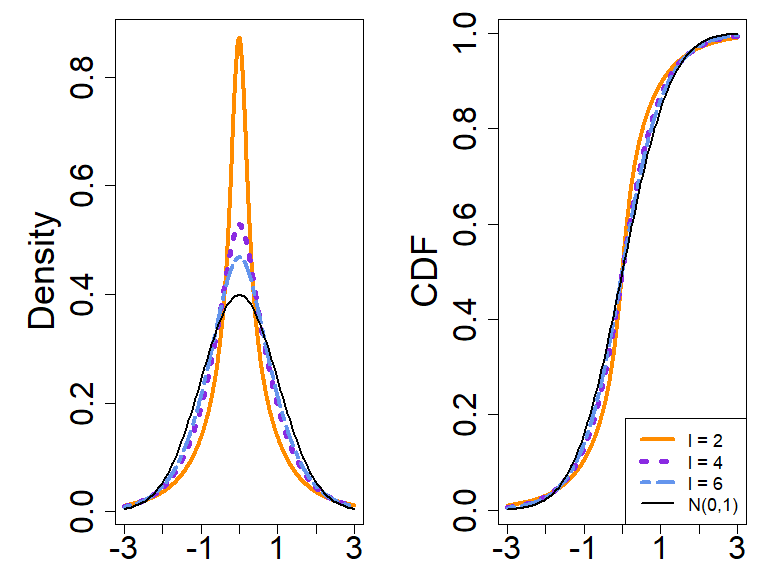}
        \caption{Density (left) and c.d.f.\  (right) of $S^{(\ell)}$ for fixed $r = 0.99$ and varying $\ell$ ($\ell = 2, 4, 6$), compared to those of a $N(0,1)$. This illustrates that $S^{(\ell)}$ converges to a $N(0,1)$ as $\ell$ grows.}
        \label{cdf_df_ell_comparison}
    \end{figure}

 Lastly, the first moments of $S^{(\ell)}$ (obtained with simple calculations in \texttt{Mathematica}) are
    \begin{equation}
        \E[S^{(\ell)}] = 0, \quad \E[(S^{(\ell)})^2] = 1, \quad \E[(S^{(\ell)})^3] = 0 \quad \text{and} \quad \E[(S^{(\ell)})^4] = \frac{6 r^4}{\ell - 1} + 3.
    \end{equation}
    Thus, an upper bound on the kurtosis of $S^{(\ell)}$ is $6 / (\ell - 1) + 3$, which implies that the limiting r.v.\ $S^{(\ell)}$ can be substantially more heavy-tailed than the standard Gaussian distribution (which is also seen in Figure~\ref{cdf_df_r2_comparison}).

\subsection{Second example}\label{sec:example.2}

Consider the sequence of graphs $\{G_m\}_{m \geq 1}$ as displayed in Figure~\ref{fig:second.graph} for $m=6$, where $M_0, M_1, M_2, \dots, M_{m+1}$ is a sequence of i.i.d.\ \textup{Bernoulli}\hspace{0.3mm}(1/2) \text{r.v.s} assigned to the vertices. For each $m$, the graph $G_m$ has $v(m)= m+2$ vertices and $n=2m$ edges. Every vertex in the set $\{M_1,M_2,\dots,M_m\}$ (in the middle) is linked by an edge to the adjacent vertices $M_0$ (on the left) and $M_{m+1}$ (on the right). This sequence of graphs yields Theorem~\ref{limiting_Sn_thm.2nd.example}.

\begin{figure}[ht]
    \centering
    \includegraphics[width=60mm]{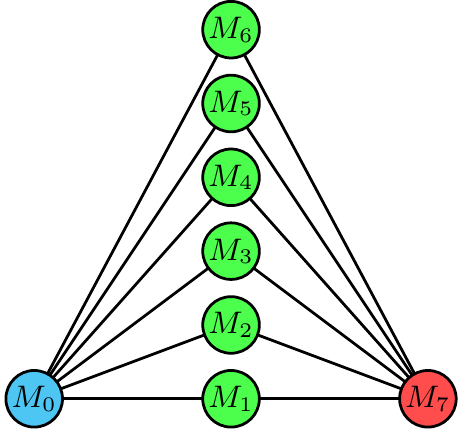}
    \caption{Illustration of the graph $G_6$ in our second example.}
    \label{fig:second.graph}
\end{figure}

\begin{thm}\label{limiting_Sn_thm.2nd.example}
Let $\{G_m\}_{m \geq 1}$ be the sequence of graphs described above and consider the construction from Section~\ref{sec:construction} where Condition~\ref{cond:F} is satisfied with $\ell=2$. Then,
    \begin{align}
        \xi_n
        &\stackrel{\mathrm{law}}{\longrightarrow}  \sqrt{2}I \cdot Z, \quad \text{as $m\to \infty$ (and thus as $n\to \infty$),}
    \end{align}
where the random variables $I\sim \textup{Bernoulli}\hspace{0.3mm}(1/2)$ and $Z\sim N(0,1)$ are independent.
\end{thm}

\begin{proof}
If $I\sim \textup{Bernoulli}\hspace{0.3mm}(1/2)$ and $B \sim \textup{Binomial}\hspace{0.3mm}(m, 1/2)$ are independent r.v.s, then $\Xi_n$ satisfies
\begin{equation}\label{eq:second.example.representation}
    \Xi_n \stackrel{\mathrm{law}}{=} I \cdot 2 B + (1 - I) \cdot m.
\end{equation}
Indeed, if the Bernoulli r.v.s $M_0$ and $M_{m+1}$ are equal (this is represented by $I = 1$ in~\eqref{eq:second.example.representation}, which has probability~$1/2$), then for every vertex $M_1,M_2,\dots,M_m$ in the middle, the sum of the $1$'s on the two adjacent edges will be $2$ with probability $1/2$ and $0$ with probability~$1/2$. By the independence of the Bernoulli r.v.s $M_1,M_2,\dots,M_m$, we can thus represent the sum of the ``$m$ sums of $1$'s'' that we just described by $2 B$ where $B\sim \text{Binomial}\hspace{0.2mm}(m,1/2)$. Similarly, if the Bernoulli r.v.s $M_0$ and $M_{m+1}$ are not equal (this is represented by $I = 0$ in \eqref{eq:second.example.representation}, which has probability~$1/2$), then for every vertex $M_1,M_2,\dots,M_m$ in the middle, the sum of the $1$'s on the two adjacent edges will always be $1$ (either the left edge is $1$ and the right edge is $0$, or vice-versa, depending on whether $(M_0 = 1,M_{m+1}=0)$ or $(M_0 = 0,M_{m+1}=1)$). Since there are $m$ vertices in the middle when $I = 0$, the total sum of the $1$'s on the edges is always $m$. By combining the cases $I = 1$ and $I = 0$, we get the representation \eqref{eq:second.example.representation}.

Lastly, here $\E[\Xi_n] = m$ and $\Var(\Xi_n) = \frac{m}{2}$ so that, by L\'evy's continuity theorem,

\begin{equation}
\xi_n = \frac{\Xi_n - m}{\sqrt{\frac{m}{2}}} = \sqrt{2} I \cdot \frac{B - m/2}{\sqrt{\frac{m}{4}}} \stackrel{\mathrm{law}}{\longrightarrow} \sqrt{2} I \cdot Z, \quad \text{where } Z\sim N(0,1).
\end{equation}
This ends the proof.
\end{proof}

\begin{rem}
By Theorem~\ref{limiting_Sn_thm}, $S_n$ converges in law to a random variable:
 \begin{equation}\label{eq:limit.S.example.2}
                S \leqdef \sqrt{1 - r^2} Z_1 + r \sqrt{2} I Z_2,
            \end{equation}
            where the random variables $Z_1,Z_2\sim N(0,1)$ and $I\sim \textup{Bernoulli}\hspace{0.3mm}(1/2)$ are all independent, and $r \leqdef \frac{\mu_V - \mu_U}{2\sigma}$. Simple calculations then yield
    \begin{equation}
        \E[S] = 0, \quad \E[S^2] = 1, \quad \E[S^3] = 0 \quad \text{and} \quad \E[S^4] = 3 (1 + r^4),
    \end{equation}
    so that $S$ in \eqref{eq:limit.S.example.2} is always heavier tailed than a standard Gaussian r.v.\ (provided $r \neq 0$, which is not a stringent requirement as seen in Remark~\ref{rem:main.thm.1}).
\end{rem}

\subsection{Third example}\label{sec:example.3}
In our construction, a CLT can hold. As a `positive example', we consider here the sequence of $m$-hypercube graphs, which have $v(m)=2^m$ vertices and $n=m 2^{m-1}$ edges. Despite being `highly connected’, these graphs \textit{do} induce a Gaussian limit for $\{S_n\}_{n\geq1}$.

\begin{thm}\label{limiting_Sn_thm.hypercube.example}
Let $\{G_m\}_{m\geq 1}$ be the sequence of $m$-hypercube graphs and consider the construction from Section~\ref{sec:construction} where Condition~\ref{cond:F} is satisfied with $\ell=2$. Then, $\xi_n$ is asymptotically Gaussian as $m\to \infty$ (and thus as $n\to \infty$).
\end{thm}

\begin{proof}

First, note that each vertex of $G_m$ can be represented by a binary vector of $m$ components. To be clear here, the hypercube graphs are all embedded in the same infinite dimensional hypercube graph, and the same goes for the Bernoulli r.v.s $M_1, M_2, \dots, M_{2^m}$ assigned to the vertices. By definition of the $m$-hypercube graph, $(i,j)$ is an edge if and only if $i$ and $j$ differ by only one binary component, which we write $i\sim j$ for short. In particular, we write $i\sim_d j$ if $i$ and $j$ differ only in the $d$-th binary component, where $1 \leq d \leq m$. With $\Xi_n$ and $D_{i,j}$ defined as in \eqref{the_D_sequence} and \eqref{eq:define.Dij}, respectively, it will be useful here to work instead with the zero-mean r.v.s, $\widetilde{\Xi}_n$ and $\widetilde{D}_{i,j}$, defined as
\vspace{-2mm}
\begin{equation}
    \widetilde{\Xi}_m = 2 \, \Xi_n - n = \sum_{i\sim j} \widetilde{D}_{i,j}, \quad \text{and} \quad \widetilde{D}_{i,j} = 2 D_{i,j} -1 =
    \begin{cases}
    1, & \textup{if } M_i = M_j, \\
    -1, & \textup{otherwise}.
    \end{cases}
\end{equation}
We will prove below that $\widetilde{\Xi}_m$ is asymptotically Gaussian, which implies that $\xi_n$ is as well. We have the following decomposition:
\begin{equation}
    \widetilde{\Xi}_m = \sum_{d=1}^m \widetilde{\Xi}_m^{(d)}, \quad \text{where} \quad \widetilde{\Xi}_m^{(d)} \leqdef \sum_{i\sim_d j} \widetilde{D}_{i,j}.
\end{equation}
Let $\mathcal{G}_d = \sigma(\widetilde{D}_{i,j} : i\sim_d j)$, and let $\mathcal{F}_m \leqdef \sigma(\cup_{d=1}^m \mathcal{G}_d)$ be the smallest $\sigma$-algebra containing the sets of all the $\mathcal{G}_d$'s, for $1 \leq d \leq m$. Then, $\mathbb{F} = \{\mathcal{F}_m\}_{m\in \N_0}$ is a filtration, where we define $\mathcal{F}_0 \leqdef \{\emptyset, \Omega\}$.
We have the following preliminary result (we complete the proof of Theorem~\ref{limiting_Sn_thm.hypercube.example} right after).

\begin{lem}\label{lem:martingale.hypercube}
    If $\widetilde{\Xi}_0 \leqdef 0$, then for every $m\in \N_0$, the process $\{\widetilde{\Xi}_k / \sqrt{\Var(\widetilde{\Xi}_m)}\}_{0 \leq k \leq m}$ is a zero-mean and bounded $\mathbb{F}$-martingale with differences $\widetilde{\Xi}_m^{(d)} / \sqrt{\Var(\widetilde{\Xi}_m)}, ~ 1 \leq d \leq m$.
\end{lem}

\begin{proof}[Proof of Lemma~\ref{lem:martingale.hypercube}]
    The process $\{\widetilde{\Xi}_m\}_{m\in \N_0}$ is trivially $\mathbb{F}$-adapted and integrable.
    To conclude that it is a $\mathbb{F}$-martingale, it is sufficient to show that
    \begin{equation}
        \E[\widetilde{\Xi}_m^{(k)} | \mathcal{F}_{k-1}] = 0, \quad \text{for all } 1 \leq k \leq m.
    \end{equation}
    By symmetry of the construction, the case $k = 1$ is trivial (i.e., $\E[\widetilde{\Xi}_m^{(1)}] = 0$). Therefore, assume that $k \geq 2$.
    Consider any instance $\omega\in \Omega$ for the values of the Bernoulli r.v.s on the vertices of the $m$-hypercube such that $\sum_{d=1}^{k-1} \widetilde{\Xi}_m^{(d)} (\omega) = s$ and $\widetilde{\Xi}_m^{(k)}(\omega) = t$, where $s,t$ are any specific integer values.
    For every such instance $\omega$, there exists a `conjugate' instance $\overline{\omega}$ where $\sum_{d=1}^{k-1} \widetilde{\Xi}_m^{(d)} (\overline{\omega}) = s$ and $\widetilde{\Xi}_m^{(k)}(\overline{\omega}) = - t$. Indeed, take the configuration $\omega$, then for every vertex that has its $k$-th binary component equal to $1$, flip the result of the Bernoulli r.v.\ ($0$ under $\omega$ becomes $1$ under $\overline{\omega}$, and $1$ under $\omega$ becomes $0$ under $\overline{\omega}$). Since the Bernoulli r.v.s on the vertices are i.i.d., and the values $0$ and $1$ are equiprobable, note that $\PP(\{\omega\} | \mathcal{F}_{k-1})(u) = \PP(\{\overline{\omega}\} | \mathcal{F}_{k-1})(u)$ for all $u\in \Omega$ such that $\sum_{d=1}^{k-1} \widetilde{\Xi}_m^{(d)} (u) = s$.
    Therefore, for any summand of the form $\widetilde{\Xi}_m^{(k)}(\omega) \cdot \PP(\{\omega\} | \mathcal{F}_{k-1})(u)$ in the calculation of $\E[\widetilde{\Xi}_m^{(k)} | \mathcal{F}_{k-1}](u)$, it will always be cancelled by $\widetilde{\Xi}_m^{(k)}(\overline{\omega}) \cdot \PP(\{\overline{\omega}\} | \mathcal{F}_{k-1})(u)$.
    Since we assumed nothing on $s$, we must conclude that $\E[\widetilde{\Xi}_m^{(k)} | \mathcal{F}_{k-1}] = 0$.
\end{proof}

Aside from Lemma~\ref{lem:martingale.hypercube}, we also have the following three properties related to the increments of the process $\big\{\widetilde{\Xi}_k / \sqrt{\Var(\widetilde{\Xi}_m)}\big\}_{0 \leq k \leq m}$:
\begin{enumerate}[(a)]
    \item $\max_{1 \leq d \leq m} \frac{\widetilde{\Xi}_m^{(d)}}{\sqrt{\Var(\widetilde{\Xi}_m)}} \stackrel{\PP}{\longrightarrow} 0$. Indeed, by a union bound and Markov's inequality with exponent $4$, we have, for any $\varepsilon > 0$,
    \begin{equation*}
        \begin{aligned}
            \PP\Bigg(\max_{1 \leq d \leq m} \bigg|\frac{\widetilde{\Xi}_m^{(d)}}{\sqrt{\Var(\widetilde{\Xi}_m)}}\bigg| > \varepsilon\Bigg) &\leq m \cdot \PP\Bigg(\bigg|\frac{\widetilde{\Xi}_m^{(1)}}{\sqrt{\Var(\widetilde{\Xi}_m)}}\bigg| > \varepsilon\Bigg) \leq m \cdot \frac{\E\big[(\widetilde{\Xi}_m^{(1)})^4\big]}{\varepsilon^4 m^2 \big(\E\big[(\widetilde{\Xi}_m^{(1)})^2\big]\big)^2} \leq \frac{C}{\varepsilon^4 m} \stackrel{m\to\infty}{\longrightarrow} 0,
        \end{aligned}
    \end{equation*}
    where $C > 0$ is a universal constant.
    \item By the weak law of large numbers for weakly correlated r.v.s with finite variance, and the fact that $\Var(\widetilde{\Xi}_m) = m \Var(\widetilde{\Xi}_m^{(d)}) = m \, \E[(\widetilde{\Xi}_m^{(d)})^2]$ for all $1 \leq d \leq m$, we have
        \begin{equation*}
            \sum_{d=1}^m \frac{(\widetilde{\Xi}_m^{(d)})^2}{\Var(\widetilde{\Xi}_m)} = \frac{1}{m} \sum_{d=1}^m \frac{(\widetilde{\Xi}_m^{(d)})^2}{\E[(\widetilde{\Xi}_m^{(d)})^2]} \stackrel{\PP}{\longrightarrow} 1, \quad \text{as } m\to \infty.
        \end{equation*}
    \item $\E\big[\max_{1 \leq d \leq m} \frac{(\widetilde{\Xi}_m^{(d)})^2}{\Var(\widetilde{\Xi}_m)}\big]$ is bounded in $m$. Indeed,
    \begin{equation*}
        \E\bigg[\max_{1 \leq d \leq m} \frac{(\widetilde{\Xi}_m^{(d)})^2}{\Var(\widetilde{\Xi}_m)}\bigg] \leq \frac{\E\big[\sum_{d=1}^m (\widetilde{\Xi}_m^{(d)})^2\big]}{\Var(\widetilde{\Xi}_m)} = \frac{\Var(\widetilde{\Xi}_m)}{\Var(\widetilde{\Xi}_m)} = 1 < \infty.
    \end{equation*}
\end{enumerate}
By Lemma~\ref{lem:martingale.hypercube}, $(a)$, $(b)$, $(c)$, and the central limit theorem for martingale arrays \citep[Theorem~3.2]{MR624435}, we conclude that
\begin{equation}\label{eq:martingale.limit.CLT.hypercube}
    \frac{\widetilde{\Xi}_m}{\sqrt{\Var(\widetilde{\Xi}_m)}} \stackrel{\mathrm{law}}{\longrightarrow} N(0,1), \quad \text{as } m\to \infty.%
    \footnote{Approximately four days after we came up with this proof, Yuval Peres provided an interesting and completely different proof of \eqref{eq:martingale.limit.CLT.hypercube} (not using martingales) in the following MathStackExchange post:

    \noindent
    \href{https://math.stackexchange.com/questions/3993902/central-limit-theorem-for-dependent-bernoulli-random-variables-on-the-edges-of-a}{https://math.stackexchange.com/questions/3993902/central-limit-theorem-for-dependent-bernoulli-random-variables-on-the-edges-of-a}.}
\end{equation}
This ends the proof of Theorem~\ref{limiting_Sn_thm.hypercube.example}.
\end{proof}

\subsection{Fourth example}



Figure~\ref{fig:third.graph} shows a graph which can easily be made arbitrarily large (displayed here for $m=6$). We have the following theorem.

\begin{thm}\label{limiting_Sn_thm.3nd.example}
    Consider the construction from Section~\ref{sec:construction} where Condition~\ref{cond:F} is satisfied with $\ell=2$. Let the graphs $G_m$ be defined as described in the caption of Figure~\ref{fig:third.graph}. Then, $\xi_n$ is asymptotically Gaussian.
\end{thm}

\begin{figure}[H]
    \centering
    \includegraphics[width=100mm]{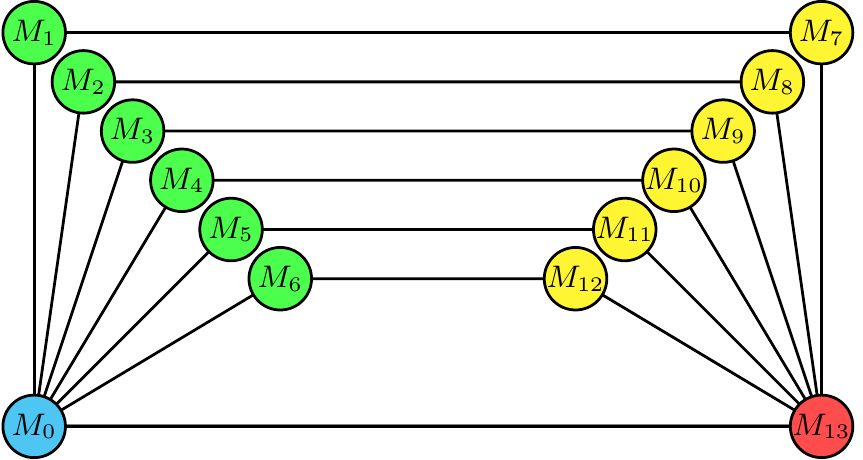}
    \caption{Illustration for $m = 6$ of the general construction where the vertex $M_0$ (on the bottom left) is linked by an edge to $M_{2m+1}$ (on the bottom right), every vertex in the set $\{M_1,M_2,\dots,M_m\}$ (on the top left) is linked by an edge to the vertex $M_0$ (on the bottom left), every vertex in the set $\{M_{m+1},M_{m+2},\dots,M_{2m}\}$ (on the top right) is linked by an edge to the vertex $M_{2m+1}$ (on the bottom right), and $M_i$ (on the top left) is linked by an edge to $M_{m+i}$ (on the top right) for all $1 \leq i \leq m$.}
    \label{fig:third.graph}
\end{figure}

\begin{proof}
If $I\sim \textup{Bernoulli}\hspace{0.3mm}(1/2)$ and $B \sim \textup{Binomial}\hspace{0.3mm}(m, 1/4)$ are independent r.v.s, then the number of $1$'s on the edges satisfies
\begin{equation}\label{eq:third.example.representation}
    \Xi_n \stackrel{\mathrm{law}}{=} I \cdot (1 + m + 2 B) + (1 - I) \cdot 2 (m - B).
\end{equation}
Indeed, if the Bernoulli r.v.s $M_0$ and $M_{2m+1}$ are equal in Figure~\ref{fig:third.graph} (this is represented by $I = 1$ in \eqref{eq:third.example.representation}, which has probability $1/2$), then for each of the $m$ $4$-cycles in the graph, the sum of the $1$'s on the left, top and right edges will be $3$ with probability $1/4$ and $1$ with probability $3/4$. By the independence of the Bernoulli r.v.s on the top-left and top-right corners of the $4$-cycles, we can thus represent the sum of the ``$m$ sums of $1$'s'' that we just described by $m + 2 B$ where $B\sim \text{Binomial}\hspace{0.2mm}(m,1/4)$. We get $1 + m + 2 B$ by including the `$1$' for the bottom edge $(M_0,M_{2m+1})$, which we only count once since this edge is common to all the $4$-cycles. Similarly, if the Bernoulli r.v.s $M_0$ and $M_{2m+1}$ are not equal in Figure~\ref{fig:third.graph} (this is represented by $I = 0$ in \eqref{eq:third.example.representation}, which has probability $1/2$), then for each of the $m$ $4$-cycles in the graph, the sum of the $1$'s on the left, top and right edges will be $2$ with probability $3/4$ and $0$ with probability $1/4$. By the independence of the Bernoulli r.v.s on the top-left and top-right corners of the $4$-cycles, we can thus represent the sum of the ``$m$ sums of $1$'s'' that we just described by $2 (m - B)$ since $m - B\sim \text{Binomial}\hspace{0.2mm}(m,3/4)$.
 By combining the cases $I = 1$ and $I = 0$, we get the representation \eqref{eq:third.example.representation}.

Easy calculations then yield
\begin{equation}
    \E[\Xi_n] = \frac{3m}{2} + \frac{1}{2} \quad \text{and} \quad \Var(\Xi_n) = \frac{3m}{4} + \frac{1}{4}.
\end{equation}
Hence, by L\'evy's continuity theorem,
\begin{equation}
    \xi_n = \frac{\Xi_n - \big(\frac{3m}{2} + \frac{1}{2}\big)}{\sqrt{\frac{3m}{4} + \frac{1}{4}}} = \frac{(2 I - 1) \cdot 2 (B - \frac{m}{4}) + (I - \frac{1}{2})}{\sqrt{\frac{3m}{4} + \frac{1}{4}}} \stackrel{\mathrm{law}}{\longrightarrow} (2 I - 1) \cdot W \stackrel{\mathrm{law}}{=} Z,
\end{equation}
where $W,Z\sim N(0,1)$.
\end{proof}


\section{The general case \texorpdfstring{$K\geq 4$}{K >= 4}} \label{sec:why.cannot.generalize}
One can easily adapt the methodology presented in this paper to build new sequences of $K$-tuplewise independent random variables (with an arbitrary margin $F$).  Indeed, all one needs to do is find a growing sequence of simple graphs of girth $K+1\geq 5$ and then, as before, put i.i.d.\ discrete uniforms on the vertices and assign $1$'s to edges for which the r.v.s on the adjacent vertices are equal. A girth of $K+1$ guarantees $K$-tuplewise independence of the sequences hence created. An arbitrary margin $F$ can be obtained as before by defining sequences $\{U_j\}_{j \geq 1}$ and $\{V_j\}_{j \geq 1}$ as in \eqref{eq:U_V_sequence}, and then creating the final sequence $\{X_j\}_{j \geq 1}$ as in \eqref{the_X_sequence}.

Whether or not sequences created this way will satisfy a CLT is a different (and difficult) question. In \citep{Balbuena_2008}, the author constructs explicitly an infinite collection of simple connected regular graphs of girth $6$ and diameter $3$, which we denote by $G_q$, where the index $q$ runs over the possible prime powers. These graphs are obtained as the incidence graphs of projective planes of order $q = k-1$. For any given prime power $q$, the graph $G_q$ is $(q+1)$-regular and has $2 \cdot (q^2 + q + 1)$ vertices. In particular, it is a $(k,6)$-cage because the number of vertices achieves the Moore (lower) bound, see, e.g., \citet[Chapter 23]{Biggs_1993}. This extremely uncommon sequence of graphs would be the perfect candidate for our construction to display a limiting non-Gaussian law for the normalized sum $S_n$. Indeed, in addition to having a minimal number of vertices, these graphs $G_q$ also have a constant (and finite) diameter, which means that we do not  have strong mixing of the binary random variables $D_j$ assigned to the edges (strong mixing is the most common assumption for a CLT with dependent random variables, see, e.g., \citet{Rosenblatt_1956}). However, even in this context where the edges' dependence is, in a sense, maximized (because of the constant diameter and the minimal number of vertices), our simulations show that we cannot reject the hypothesis of a Gaussian limit for $S$. We applied the following normality tests with $q = 2^6$ (which corresponds to a sample of size $n=(q+1)(q^2+q+1) = 270,\!465)$ and $5,\!000$ samples:
\begin{table}[ht]
    \centering
    \begin{tabular}{c|ccc}
        test & Shapiro-Wilk & Anderson-Darling & Pearson chi-square \\ \hline
        test statistic & 0.9997 & 0.2993 & 67.9360 \\[-0.5mm]
        p-value & 0.7148 & 0.5846 & ~0.7602 \\
    \end{tabular}
\end{table}

\noindent
For the interested reader, the code is provided in Appendix~\ref{appendix_computing_code}.

\begin{rem}\label{rem:link.graph.theory}
    There seems to be a link between the fact that examples of asymptotic non-normality of $\{S_n\}_{n\geq 1}$ exist for $K \leq 3$ (girth $g \leq 4$) but not for $K \geq 4$ (girth $g \geq 5$), and the fact that there exists growing sequences of regular graphs $G_m$ of girth $g \leq 4$ where \begin{equation}\label{eq:high_connectivity}
        \liminf_{m\to \infty} \frac{\mathrm{degree}(G_m)}{\text{\# of  vertices of } G_m} > 0,
    \end{equation}
    (the $\liminf_{n\to \infty}$ here is certainly a measure of the connectivity of the graphs $G_m$'s), whereas we always have
    \begin{equation*}
        \lim_{m\to \infty} \frac{\mathrm{degree}(G_m)}{\text{\# of  vertices of } G_m}= 0,
    \end{equation*}
    for regular graphs of girth $g\geq 5$, see, e.g., \citet[Proposition 23.1]{Biggs_1993}. This dichotomy in the statistics context (and its link to graph theory) seems to be a completely new and promising observation.
\end{rem}

\begin{rem}
In contrast to the sequence of graphs in our first example (Section~\ref{sec:example.1}), the sequence of hypercube graphs in our third example (Section~\ref{sec:example.3}) \textit{do not} satisfy \eqref{eq:high_connectivity}.
The property \eqref{eq:high_connectivity} in a sense measures the connectivity of the graphs, and therefore the level of dependence between the r.v.s $D_{i,j}$ assigned to the edges in our construction. Since \eqref{eq:high_connectivity} \textit{cannot} be satisfied for $K \geq 4$ when the underlying graphs are regular, the third example reinforces our intuition that, for $K\geq 4$, the sequence $\{\xi_n\}_{n\geq 1}$ (and thus $S_n$) will always converge to a Gaussian random variable.
\end{rem}

\appendix

    \section{The variance-gamma distribution}\label{sec:variance.gamma}

    \begin{defin}\label{def:variance.gamma.distribution}
        The variance-gamma distribution with parameters $\alpha > 0$, $\theta\in \R$, $s > 0$, $c\in \R$ has the density function
        \begin{equation}
            f(x) \leqdef \frac{1}{s \sqrt{\pi} \hspace{0.2mm}\Gamma(\alpha/2)} e^{\frac{\theta}{\alpha^2}(x - c)} \bigg(\frac{|x - c|}{2 \sqrt{\theta^2 + s^2}}\bigg)^{\hspace{-1mm}\frac{\alpha-1}{2}} K_{\frac{\alpha-1}{2}}\bigg(\frac{\sqrt{\theta^2 + s^2}}{s^2} |x - c|\bigg), \quad x\in \R,
        \end{equation}
        where $K_{\nu}$ is the modified Bessel function of the second kind of order $\nu$.
        If a certain random variable $X$ has this distribution, then we write $X\sim \mathrm{VG}(\alpha,\theta,s^2,c)$.
    \end{defin}

    We have the following result, which is a consequence (for example) of Theorem 1 in \citep{MR3942099}.

    \begin{lem}\label{lem:variance.gamma.result}
        Let $W_1, W_2, \dots, W_n\stackrel{\textup{i.i.d.}}{\sim} N(0,s^2)$ and $Z_1, Z_2, \dots, Z_n \stackrel{\textup{i.i.d.}}{\sim} N(0,s^2)$ be two independent sequences, then $Q_n \leqdef \sum_{i=1}^n W_i Z_i\sim \mathrm{VG}(n,0,s^2,0)$, following Definition~\ref{def:variance.gamma.distribution}, and the density function of $Q_n$ is given by
        \begin{equation}
            f_{Q_n}(x) = \frac{1}{s^2 \sqrt{\pi} \hspace{0.2mm}\Gamma(n/2)} \bigg(\frac{|x|}{2 s^2}\bigg)^{\hspace{-1mm}\frac{n-1}{2}} K_{\frac{n-1}{2}}\bigg(\frac{|x|}{s^2}\bigg), \quad x\in \R.
        \end{equation}
        It is easy to verify that the characteristic function of $Q_n$ is given by
        \begin{equation}\label{eq:characteristic.variance.gamma.result}
            \varphi_{Q_n}(t) = (1 + s^4 t^2)^{-n/2}, \quad t\in \R,
        \end{equation}
        and the expectation and variance are given by
        \begin{equation}
            \E[Q_n] = 0 \quad \text{and} \quad \Var[Q_n] = n \hspace{0.3mm} s^4.
        \end{equation}
    \end{lem}

    \section{Computing codes}\label{appendix_computing_code}

    The published version of this article (DOI: https://doi.org/10.1515/demo-2021-0120) provides the computing \texttt{R} codes as supplementary material.

\vspace{5mm}
\noindent
{\bf Acknowledgements:} The authors are indebted to two anonymous referees, whose comments led to significant improvements of the manuscript.
G.\ B.\ B.\ acknowledges financial support from UNSW Sydney under a University International Postgraduate Award, from UNSW Business School under a supplementary scholarship, and from the FRQNT (B2).
F.\ O.\ is supported by postdoctoral fellowships from the NSERC (PDF) and the FRQNT (B3X supplement and B3XR). This research includes computations using the computational cluster Katana supported by Research Technology Services at UNSW Sydney.

\vspace{5mm}
\noindent
{\bf Conflict of interest:} The authors declare no conflict of interest.

%
%


\end{document}